\documentclass[11pt,letterpaper]{amsart}
\usepackage[utf8]{inputenc}
\usepackage[english]{babel}
\usepackage{amsmath}
\usepackage{amsfonts}
\usepackage{amssymb}
\usepackage{amsthm}
\usepackage{tikz-cd} 

\usepackage[title]{appendix}
\usepackage[]{algorithm2e}

\usepackage{hyperref}

\usepackage[normalem]{ulem}
\useunder{\uline}{\ul}{}
\theoremstyle{plain}
\newtheorem{teo}{}[section]
\newtheorem{prop}[teo]{Proposition}
\newtheorem{cor}[teo]{Corollary}

\newtheorem{thm}[teo]{Theorem}
\theoremstyle{definition}
\newtheorem{ex}[teo]{Example}
\newtheorem{rem}[teo]{Remark}
\newtheorem{df}[teo]{Definition}

\newcommand\blfootnote[1]{%
  \begingroup
  \renewcommand\thefootnote{}\footnote{#1}%
  \addtocounter{footnote}{-1}%
  \endgroup
}

\definecolor{cof}{RGB}{219,144,71}
\definecolor{pur}{RGB}{186,146,162}
\definecolor{greeo}{RGB}{91,173,69}
\definecolor{greet}{RGB}{52,111,72}

\usepackage{graphicx}

\title{Matrix Invariants as Homotopy Invariants in Finite $T_0$-spaces }
\author{Pedro J. Chocano}
\date{}

\begin{document}
\maketitle

\begin{abstract}
We establish a bijection between the set of finite topological $T_0$-spaces (or  partially ordered sets) and equivalence classes of square matrices. The absolute value of the determinant or the rank of these matrices serve as simple homotopy invariants for the corresponding topological spaces, and consequently, for finite simplicial complexes. To conclude, we explore further relationships and problems concerning finite posets within the context of these matrices.
\end{abstract}

\section{Introduction}\label{sec:introduccion}
\blfootnote{2020  Mathematics  Subject  Classification:  	06A11, 57Q10, 55P10, 05B20.}
\blfootnote{Keywords: Posets, Alexandroff spaces, simplicial complexes, homotopy, simple homotopy, matrices.}
%\blfootnote{This research is partially supported by Grant PGC2018-098321-B-100 from Ministerio de Ciencia, Innovación y Universidades (Spain).}

The objective of this paper is to associate each finite topological $T_0$-space--or equivalently, each partially ordered set (poset)--with a square matrix, potentially non-invertible, that encodes homotopical information. While connections between posets and matrices have been previously explored (see \cite{ivariant1994Ochiai,structure1994Ochiai,quasis1966sharp,stong1966finite}), to the best of the author’s knowledge, none of these approaches utilize natural matrix invariants to derive results about homotopy. This limitation stems from the fact that the matrices considered in those works are typically congruent to lower triangular matrices with 1’s on the diagonal.

There are, however, well-established associations between finite topological $T_0$-spaces and simplicial complexes that have proven fruitful in algebraic topology. Notably, there exists a functor $\mathcal{K}$ from the category of finite topological $T_0$-spaces to the category of finite simplicial complexes, and a functor $\mathcal{X}$ in the reverse direction (see \cite{barmak2011algebraic,may1966finite}).These constructions not only allow for the translation of topological questions into combinatorial ones, but also enable the definition of new invariants for simplicial complexes through the study of finite topological spaces.

The notion of simple homotopy, introduced by J.H.C. Whitehead for simplicial complexes \cite{whitehead1939simplicial,whitehead1941onincidence,whitehead1950simple}, provides a geometric refinement of homotopy theory (see also \cite{cohen1970course}). Although homotopy and simple homotopy are not generally equivalent, under certain conditions they coincide: for simplicial complexes $K$ and $L$, a homotopy equivalence $f : |K| \to |L|$ is a simple homotopy equivalence if and only if the Whitehead group of $K$ is trivial. More recently, J.A. Barmak and E.G. Minian adapted this concept to finite topological $T_0$-spaces  \cite{barmak2008simple}, showing that the functors $\mathcal{K}$ and $\mathcal{X}$ induce a bijection between simple equivalence classes of finite topological $T_0$-spaces and simple homotopy types of finite simplicial complexes (\cite[Corollary 3.11]{barmak2008simple}).

In this work, we introduce a class of matrices associated with finite topological $T_0$-spaces and show that the absolute value of their determinant is a simple homotopy invariant for these spaces. Moreover, using the functors $\mathcal{K}$ and $\mathcal{X}$ together with the results of \cite{barmak2008simple}, this invariant also applies to finite simplicial complexes. It is worth noting that an earlier draft of this paper was anonymously refereed for a journal, and one of the comments based on the examples of the previous draft conjectured whether this number coincides with the reduced Euler characteristic, up to sign. Recently, this conjecture has been confirmed (see \cite{chocano2025reduced}); that is, the determinant of these matrices corresponds to the reduced Euler characteristic of the associated finite space. Consequently, some results in this paper—specifically those concerning the homotopy invariance of the determinant—can be regarded as an alternative proof of the fact that the Euler characteristic is a homotopy invariant. In this case, the proof makes use exclusively of linear algebraic combinatorial techniques.

Moreover, we show that the difference between the order of the matrix and its rank also constitutes a simple homotopy invariant. In contrast, other algebraic invariants--such as the characteristic polynomial--are shown to be topological but not homotopy invariants as well as eigenvalues. Additionaly, we find some connections between these matrices and notions from posets and graphs.

The reader is assumed to be familiar with the theory of Alexandroff spaces; for general background, we refer to \cite{barmak2011algebraic,may1966finite}. Nevertheless, basic definitions and notation are recalled in Section~\ref{sec:preliminaries}. In Section~\ref{sec:finite_matrix_graphs}, we establish a bijection between posets and matrices, relate them to a class of graphs, and examine their structural properties. Section~\ref{sec:invariantes} focuses on matrix invariants, including the determinant, rank, and characteristic polynomial. Finally, Section~\ref{sec:other_aspects} discusses further connections between finite spaces and the matrices introduced earlier.
\section{Preliminaries}\label{sec:preliminaries}
Let $X$ be a finite $T_0$-space and $x \in X$. Define:
\begin{itemize}
    \item $U_x$ as the intersection of all open sets containing $x$,
    \item $F_x$ as the intersection of all closed sets containing $x$,
    \item $\hat{U}_x = U_x \setminus \{x\}$,
    \item $\hat{F}_x = F_x \setminus \{x\}$,
    \item $C_x = U_x \cap F_x$.
\end{itemize}
Set $x \leq y$ if and only if $U_x \subseteq U_y$ for $x, y \in X$. Then $(X, \leq)$ is a partially ordered set (poset). Conversely, given a finite poset $(X, \leq)$, the set of lower sets (i.e., subsets $S \subseteq X$ such that if $y \leq x$ and $x \in S$, then $y \in S$) forms a basis for a $T_0$-topology on $X$. These two constructions are mutually inverse (see \cite{alexandroff1937diskrete}). For a finite poset $X$ with $n$ points, we fix a labelling $X=\{ x_i\}_{i=1,...,n}$. For simplicity, we denote $X=\{x_i\}$ when no confusion arises.

Notice that given a finite $T_0$-space, one can also consider $x\leq_o y$ if and only if $U_y\subseteq U_x$. This relation is a partial order on $X$. Analogously, given a finite poset $(X,\leq)$ the set of upper sets (recall that an upper set $S\subseteq X$ is a set satisfying that if $x\leq y$ and $x\in S$ then $y\in S$) forms a basis for a $T_0$-topology on $X$.  When we consider this construction, we say that the poset $(X,\leq_o)$ has the opposite order and the topology of $X$ is the opposite topology.

Moreover, a map $f: X \to Y$ between finite $T_0$-spaces is continuous if and only if it is order-preserving. Therefore, the category of finite $T_0$-spaces is isomorphic to the category of finite posets. From now on, we will not distinguish between finite posets and finite $T_0$-spaces.

Let $X$ be a finite poset. The \emph{height} of $X$, denoted $\mathrm{ht}(X)$, is one less than the maximum number of elements in a chain of $X$. The height of a point $x \in X$ is defined as the height of $U_x$. The \emph{width} of $X$, denoted $\mathrm{width}(X)$, is the maximum size of an antichain in $X$. 

We write $x \prec y$ if $x < y$ and there is no $z$ such that $x < z < y$. The Hasse diagram of $X$ is a directed graph whose vertices are the elements of $X$, and there is an edge from $x$ to $y$ if $x \prec y$.

We now recall key concepts for studying the homotopy and weak homotopy types of finite topological spaces.

\begin{df} A point $x$ in a finite $T_0$-space $X$ is a down (up) beat point if $\hat{U}_x$ has a maximum ($\hat{F}_x$ has a minimum).
\end{df}

\begin{rem}\label{rem:up_son_down_opuesto} Let $X$ be a finite $T_0$-space. An up (down) beat point $x$ in $X$ is a down (up) beat point in $X$ with the opposite topology.
\end{rem}

\begin{thm}[\cite{stong1966finite}] Let $X$ be a finite $T_0$-space and let $x\in X$ be a beat point. Then $X\setminus \{x\}$ is a strong deformation retract of $X$. 
\end{thm}

\begin{df} A finite $T_0$-space is a minimal finite space if it has no beat points. A core of a finite $T_0$-space is a strong deformation retract which is a minimal finite space.
\end{df}

Note that every finite $T_0$-space has a core \cite[Theorem 2]{stong1966finite}.
\begin{thm}[\cite{stong1966finite}]\label{thm:stong_cores_homeomorphic} Let $X$ and $Y$ be finite $T_0$-spaces and with cores $X_1$, $Y_1$. Then $X$ is homotopy equivalent to $Y$ if and only if $X_1$ is homeomorphic to $Y_1$.
\end{thm}

A continuous map $f:X\rightarrow Y$ between finite topological spaces is  a weak homotopy equivalence if it induces isomorphism in all homotopy groups. Recall that any weak homotopy equivalence between topological spaces induces also isomorphisms in all homology groups.  Two spaces $X$ and $Y$ are weak homotopy equivalent if there exists a sequence of spaces $X=X_0,X_1,...,X_n=Y$ and there are weak homotopy equivalences $X_i\rightarrow X_{i+1}$ or $X_{i+1}\rightarrow X_i$ where $i=0,...,n-1$.

Let $X$ be a topological space, a finite space $Y$ is a finite model of $X$ if it is weak homotopy equivalent to $X$. Moreover, $Y$ is a minimal finite model if it is a finite model of minimum cardinality. 

\begin{df} A point $x$ in a finite $T_0$-space $X$ is a down (up) weak beat point if $\hat{U}_x$ is contractible ($\hat{F}_x$ is contractible).
\end{df}
\begin{thm}[\cite{barmak2008simple}] Let $X$ be a finite $T_0$-space and let $x\in X$ be a weak beat point. Then the inclusion $i:X\setminus \{x\}\rightarrow X$ is a weak homotopy equivalence. 
\end{thm}
Two finite $T_0$-spaces $X$ and $Y$ are simple homotopy equivalent if one can be obtained from the other by adding and removing weak beat points one by one. 

We recall two key construction in the theory of finite topological spaces. These constructions connect the classical theory of simplicial complexes (and its algebraic topology) with the theory of finite posets.
\begin{df} Let $X$ be a finite $T_0$-space. The order complex $\mathcal{K}(X)$ is the simplicial complex whose simplices are the nonempty chains of $X$.
\end{df}
\begin{thm}[\cite{mccord1966singular}] Let $X$ be a finite $T_0$-space. Then there exists a weak homotopy equivalence $f_X:|\mathcal{K}(X)|\rightarrow X$.
\end{thm}

\begin{df} Let $L$ be a finite simplicial complex. The face poset $\mathcal{X}(L)$ is the poset of simplices of $L$ ordered by inclusion .
\end{df}

\begin{thm}[\cite{mccord1966singular}] Let $L$ be a finite simplicial complex. Then there exists a weak homotopy equivalence $f_\mathcal{X}(L):|L|\rightarrow \mathcal{X}(L)$.
\end{thm}
There is a notion of simple homotopy--indeed, the classical one--within the context of simplicial complexes, introduced by J.H.C. Whitehead \cite{whitehead1939simplicial,whitehead1941onincidence,whitehead1950simple} (see \cite{cohen1970course} for a comprehensive introduction to this topic and its historical development). The idea is to identify certain moves—elementary collapses and expansions—on finite simplicial complexes that preserve the homotopy type. J.A. Barmak and E.G. Minian, in \cite{barmak2008simple}, showed that this notion is equivalent to the notion of simple homotopy defined before for finite posets via the functor $\mathcal{K}$. In this paper, we choose to work within this combinatorial framework, omitting the classical definition and relying on its equivalence with the poset-based approach.

Finally, we introduce some matrix notation:
\begin{itemize}
    \item $1_{n \times m}$ denotes the $n \times m$ matrix of ones,
    \item $0_{n \times m}$ denotes the $n \times m$ matrix of zeros,
    \item $I_n$ denotes the $n \times n$ identity matrix,
    \item $e_i$ denotes the vector with a 1 in the $i$-th position and zeros elsewhere.
\end{itemize}
When there is no risk of confusion, we will omit the subscripts of the matrices

\section{Finite Spaces, Matrices and Graphs}\label{sec:finite_matrix_graphs}

We define a class of matrices associated with finite $T_0$-spaces ad state some basic properties. Additionally, we interpret these matrices as adjacency matrices of certain directed graphs.

\begin{df}\label{def:matrices_zeros} Let $\mathcal{M}'$ be the set of all square matrices $(a_{i,j})$ with entries in $\{ 0,1\}$ such that:
\begin{enumerate}
\item $a_{i,i}=0,$
\item if $a_{i,j}=0$, then $a_{j,i}=1$,
\item if $a_{i,j}=a_{j,k}=0$, then $a_{i,k}=0.$ 
\end{enumerate}
\end{df}

Let $X,Y\in \mathcal{M}'$. Two matrices $X$ and  $Y$ are said to be equivalent if and only if there exists a permutation matrix $E$ such that $EAE^{-1}=B$. Let $\mathcal{M}$ denote the set of equivalence classes of elements in $\mathcal{M}'$.

Given a finite $T_0$-space $X=\{ x_i\}$, define the matrix $X_M=(x_{i,j})$ by:
\[
x_{i,j} = 
\begin{cases}
0 & \text{if } x_i \leq x_j, \\\\
1 & \text{otherwise}.
\end{cases}
\]

It is easy to verify that  $X_M\in \mathcal{M}'$. Notice that a different labelling of the elements of $X$ produces the same element in $\mathcal{M}$. From now on, we denote to the $i$-th column of $X_M$ by $c_i$, while the $i$-th row of $X_M$ by $r_i$. We will also denote $X_M$ as $X$ when it is clear from the context.

\begin{thm}\label{thm:matrices_grafos} The homeomorphism classes of finite $T_0$-spaces are in one-to-one correspondence with $\mathcal{M}$.
\end{thm} 
\begin{proof}

Suppose $g:X\rightarrow Y$ is a homeomorphism. If $x_i\in U_{x_j}$ ($x_i\in X\setminus U_{x_j}$), then $g(x_i)\in U_{g(x_j)}$ ($g(x_i)\in Y\setminus U_{g(x_j)}$). This shows that $X$ and $Y$ determine the same element in $\mathcal{M}$. So we have a well-defined map $T:\mathcal{F}\rightarrow \mathcal{M}$ where $\mathcal{F}$ denotes the set of finite $T_0$-spaces.

Suppose that $X=\{x_j\}$ and $Y=\{y_i\}$ are finite $T_0$-spaces such that they belong to the same class in $\mathcal{M}$. Let $E$ denote the permutation matrix satisfying $EX_ME^{-1}=Y_M$ and $\tau$ the permutation related to $E$. Then $g:X\rightarrow Y$ defined by $g(x_j)=y_{\tau(j)}$ is a homeomorphism. This gives that $T$ is injective. 

Given a matrix $M=(m_{i,j})\in \mathcal{M}$, consider $X=\{x_i\}_{i=1,...,n}$ where $n$ is the order of $M$. Set $$x_i\leq x_j\iff m_{i,j}=0.$$
Clearly, this relation is reflexive by 1, antisymmetric by 2 and transitive by 3, which shows that $X$ with this relation is a poset. It is immediate that $T(X)=M$.
  
\end{proof}
\begin{rem} By removing condition 2 in Definition \ref{def:matrices_zeros}, one may obtain a similar result for finite topological spaces that do not satisfy the $T_0$ property, that is, preordered sets. 
\end{rem}

We now provide a simple criterion to check whether a matrix with entries 0 or 1 defines a finite poset.

\begin{thm} Let $X$ be a finite set of $n$ points. A matrix $M=(m_{i,j})$ of dimensions $n\times n$ such that $m_{i,i}=0$ for every $i$ and $m_{i,j}\in \{ 0,1\}$ corresponds to a topology on $X$ if and only if $(1_{n\times n}-M)(1_{n\times n}-M)=(1_{n\times n}-M)$, where the matrix multiplication here uses Boolean arithmetic.
\end{thm}
\begin{proof}
For any given finite $T_0$-space $X$, $1_{|X|\times |X|}-X_M$ gives the standard incidence matrix of $X$ considered in \cite{quasis1966sharp}. Therefore, the result follows from \cite[Theorem 4]{quasis1966sharp}.
\end{proof}

The following result is an immediate consequence of the construction of the matrix $X_M$ for any finite $T_0$-space $X$.
\begin{prop}\label{prop:orden_opuesto_matriz} Let $X$ be a finite $T_0$-space. Then the matrix of $X$ with the opposite order is $X_M^\intercal$.
\end{prop}

We now examine some topological and combinatorial properties of $X$ and $X_M$, omitting trivial proofs.

\begin{prop}\label{prop:sumas_elementos_filas_columnas} Let $X=\{x_i\}$ be a finite $T_0$-space with $|X|=n$ and $X_M =(x_{i,j})$. Then  \begin{itemize}
\item $\sum_{k=1}^n x_{i,k}=n-|F_{x_i}|$.
\item $\sum_{k=1}^n x_{k,i}=n-|U_{x_i}|$.
\item $\sum_{k=1}^n x_{k,k}=0$.
\item $\sum_{k=1}^n x_{i,k}=0$ implies that $x_i$ is a minimum.
\item $\sum_{k=1}^n x_{k,i}=0$ implies that $x_i$ is a maximum.
\item $\sum_{k=1}^n x_{i,k}=n-1$ implies that $x_i$ is a maximal point.
\item $\sum_{k=1}^n x_{k,i}=n-1$ implies that $x_i$ is a minimal point.
\item $x_{i_1}<...<x_{i_k}$ in $X$ if and only if $\sum_{j=1}^{k-1} x_{i_j, i_{j+1}}=0$.
\end{itemize}
\end{prop}
\begin{rem} Suppose that $X=\{x_i\}$ is a finite  $T_0$-space. If $X$ has a minimum or maximum, then $\det (X_M)=0$. 
\end{rem}
%\begin{prop}\label{prop:chains_and_fences} Let $X$ be a finite $T_0$-space. 
%\begin{itemize}
%\item Then $x_{1}<...<x_{k}$ in $X$ if and only if $\sum_{j=1}^{k-1} x_{j j+1}=0$.
%\item Then $x_{1}\leq x_{2}\geq x_{3}\leq ... \geq x_{k}$ if and only if $x_{12}+x_{32}+...+x_{kk-1}=0$.
%\end{itemize} 
%\end{prop}
%\begin{proof}
%We have $x_{1}<...<x_{k}$ if and only if $x_{j}\in U_{x_{{j+1}}}$ for every $1\leq j\leq k-1$ if and only if $\sum_{j=1}^{k-1} x_{j {j+1}}=0$. The second part is similar to the first one.
%\end{proof}
\begin{prop} Let $X$ be a finite $T_0$-space. Then $X$ is not connected if and only if $X_M$ is equivalent to $$\begin{pmatrix}
Y_M & 1_{n\times m} \\
1_{m\times n} & Z_M
\end{pmatrix}, $$
where $Y_M,Z_M\in \mathcal{M}'$. 
\end{prop}
%\begin{proof}
%Suppose $X$ is not connected. Then $X=Y\sqcup Z$. Set $Y=\{x_i\}_{1\leq i \leq m}$ and $Z=\{x_i \}_{m<i\leq n}$. From this, it is trivial to deduce the result.  Suppose now that $X_M$ is equivalent to the matrix above. Consider $W=\bigcup_{i=1}^m U_{x_i}$ and $V=\bigcup_{i=m+1}^n U_{x_i}$. By hypothesis $X=W\cup V$ and $W\cap V=\emptyset$.
%\end{proof}

\begin{prop} Let $X$ be a finite $T_0$-space. Then the width of $X$ is $l$ if and only if $X_M$ is equivalent to a matrix $(x_{i,j})$ such that $x_{i,j}=1$ for every $i\neq j$ and $i,j\leq l$, and there is no $k>l$ satisfying that $x_{s,k}=x_{k,s}=1$ for every $s\leq l$.
\end{prop}
%\begin{proof}
%Suppose $width(X)=l$. Set $X=\{x_i \}$, where $\{x_1,...,x_l\}$ is an antichain of $X$. Then $X_M$ satisfies the desired conditions. We prove the other implication. The condition in the hypothesis is telling that there is an antichain with maximum size $l$, so $width(X)=l$. 
%\end{proof}

\begin{prop} Let $X=\{x_i\}$ be a finite $T_0$-space. Then $x_i\prec x_j$ if and only if $r_i+c_j^\intercal =(p_1,\cdots,p_n)$ where $p_{i}=p_j=0$ and $p_{k}\neq 0$ for every $k\neq i,j$.
\end{prop}
\begin{proof}
Suppose $x_i\prec x_j$. Then $x_{i,j}=0$ and $p_j=x_{i,j}+x_{j,j}=x_{i,i}+x_{i,j}=p_i=0$. If $x_{i,k}+x_{k,j}=0$ for some $k\neq i,j$, then $x_i<x_k$ and $x_k<x_j$, which contradicts $x_i\prec x_j$.

If $r_i+c_j^\intercal=(p_1,\cdots,p_n)$ where $p_{i}=p_j=0$ and $p_{k}\neq 0$ for every $i\neq j$, then $x_{i,j}+x_{j,j}=x_{i,i}+x_{i,j}=0$ and $x_{i,k}+x_{k,j}\neq 0$ for every $k\neq i,j$. Hence, $x_i<x_j$ and there is no $k\neq i,j$ such that $x_i\leq x_k\leq x_j$.
\end{proof}

\begin{prop} Let $X=\{x_i\}$ be a finite $T_0$-space and $X_M^2=({m}_{i,j})$. Then $m_{i,j}=|\hat{F}_{x_j}\cap \hat{U}_{x_i}|+|\hat{U}_{x_i}\cap X\setminus C_{x_j}|+|\hat{F}_{x_j}\cap X\setminus C_{x_i}|+|X\setminus C_{x_i}\cup C_{x_j}|$ and $n-m_{i,j}=|F_{x_i}\cup U_{x_j}|$.
\end{prop}
\begin{proof}
We prove the first property, the second property may be deduced from the first one. Define $K_{i,j}=\{1\leq k\leq n|x_{i,k}=x_{k,j}=1   \}$, so $m_{i,j}=|K_{i,j}|$. Let us consider $k\in K_{i,j}$. Since $x_{i,k}=1$ we get $x_i>x_k$ or $x_i\notin C_{x_k}$. Similarly, $x_{k,j}=1$ implies that $x_k>x_j$ or $x_k\notin C_{x_j}$. Hence, 
$$|K_{i,j}|= |\hat{F}_{x_j}\cap \hat{U}_{x_i}|+|\hat{U}_{x_i}\cap X\setminus C_{x_j}|+|\hat{F}_{x_j}\cap X\setminus C_{x_i}|+|X\setminus C_{x_i}\cup C_{x_j}|.$$
\end{proof}
We can argue analogously to achieve a similar result using the transpose.
\begin{prop} Let $X=\{x_i\}$ be a finite $T_0$-space. 
\begin{itemize}
\item  If  $X_MX_M^\intercal=(m_{i,j})$. Then
$m_{i,j}=|\hat{U}_{x_i}\cap \hat{U}_{x_j}|+|\hat{U}_{x_i}\cap X\setminus C_{x_j}|+ |\hat{U}_{x_j}\cap X\setminus C_{x_i}|+|X\setminus C_{x_i}\cup C_{x_j}|$
and $n-m_{i,j}=|F_{x_i}\cup F_{x_j}|. $
\item If $X_M^\intercal X_M=(m_{i,j})$. Then 
$m_{i,j}=|\hat{F}_{x_i}\cap \hat{F}_{x_j}|+|\hat{F}_{x_i}\cap X\setminus C_{x_j}|+ |\hat{F}_{x_j}\cap X\setminus C_{x_i}|+|X\setminus C_{x_i}\cup C_{x_j}|$
and $n-m_{i,j}=|U_{x_i}\cup U_{x_j}|. $
\end{itemize}
\end{prop}

Every matrix of $\mathcal{M}'$ can also be interpreted as the adjacency matrix of a digraph without loops, which we denote by $G_X$. Furthermore, we can derive a version of Theorem \ref{thm:matrices_grafos} for digraphs.  Let $\mathcal{G}$ be the set of all digraphs $G=(V,E)$ up to isomorphism such that:
\begin{enumerate}
\item There are no loops.
\item If $\{x,y\}\notin E$, then $\{y,x\}\in E$ for every $x,y\in V$.
\item If $\{x,y\},\{y,z \}\notin E$, then $\{x,z\}\notin E$  for every $x,y,z\in V$. 
\end{enumerate}

\begin{thm} The set of homeomorphism classes of finite $T_0$-spaces are in one-to-one correspondence with $\mathcal{G}$.
\end{thm}
%\begin{proof}
%The correspondence that assigns to each finite topological space $X$ the digraph $G_X$ is clearly well defined and injective. Given a digraph $G\in \mathcal{G}$, the conditions required in its definition mean that the adjacency matrix satisfies that \begin{enumerate}
%\item $x_{ii}=0$.
%\item $x_{ij}=0$, then $x_{ji}=1$.
%\item $x_{ij}=0,x_{jk}=0$, then $x_{jk}=0$.
%
%\end{enumerate}
%Therefore, the adjacency matrix belongs to $\mathcal{M}$. Applying Theorem \ref{thm:matrices_grafos} we get the bijection.
%\end{proof}

Notice that antichains of two elements can be identified in $X_M$ by simply examining entries such that $x_{i,j}=1=x_{j,i}$. This means  that  antichains of two elements correspond to paths of the form $\{x_i,x_j\}, \{ x_j,x_i\}$ in $G_X$. Similarly, antichains of $k$ elements correspond to sequences $i_1,...,i_k$ satisfying $x_{i_l,i_h}=x_{i_h,i_l}=1$ for every $1\leq h,l\leq k$. This indicates that there is a complete subgraph of $G_X$ whose vertices are $\{x_{i_1},...,x_{i_k} \}$ (or a clique). By using these observations and classical results from graph theory, we obtain the following result.
\begin{prop}\label{prop:antichains_size_2} Let $X$ be a finite $T_0$-space. Then 
\begin{itemize}
\item $\frac{\textnormal{tr}(X_M^2)}{2}$ is the number of antichains with $2$ points of $X$.
\item $\frac{\textnormal{tr}(X_M^3)}{6}$ is the number of antichains with $3$ points of $X$.
\end{itemize}
\end{prop}
Combining this result with Dylworth's theorem:
\begin{prop}\label{prop_no_antichains_size_2_is_chain} Let $X$ be a finite $T_0$-space. If $\textnormal{tr}(X_M^2)=0$, then $X$ is a totally ordered set.
\end{prop}

Edges in $G_X$ represent that there is no relation between $x_i$ and $x_j$ or $x_j>x_i$. The number of strongly connected components of $G_X$ can be interpreted as a method to ``measure'' the number of relations between the elements of the poset. For example, it is straightforward to verify that a finite totally ordered set $X$ has $|X|$ strongly connected components and the minimal finite model of the $n$-dimensional sphere $S^n$ (see \cite{barmak2007minimal}) has $n+1$ strongly connected components. This number is not a homotopy invariant as the first examples shows.

\begin{prop} Let $X$ be a finite $T_0$-space. If $X$ is not connected, then $G_X$ has only one strongly connected component.
\end{prop}
\begin{proof}
Without loss of generality we may assume that $X$ has two connected components $X^1$ and $X^2$. For every pair of vertices $x_i,x_j\in G_X$ such that $x_i\in X^1$ and $x_j\in X^2$ there is a directed edge from $x_i$ to $x_j$ and vice versa. This means that for for every pair of vertices one can find a path connecting them in both directions.
\end{proof}
\begin{rem} Let $X$ be a finite $T_0$-space. There is a bijective correspondence between cliques of $G_X$ and antichains of $X$. Hence, the problem of finding the width of a poset $X$ is equivalent to the problem of finding the maximum clique in $G_X$. 
\end{rem}

\section{Invariants}\label{sec:invariantes}
In this section, we examine the relationships between invariants of matrices and homotopy.
\subsection{Determinant}

Let $X$ be a finite $T_0$-space, and let $X_M$ denote its associated matrix. In \cite{chocano2025reduced}, the authors show that $\det(X_M)$ corresponds, up to sign, to the reduced Euler characteristic of $X$. In this section, we show that $|\det(X_M)|$ is both a homotopy invariant and a simple homotopy invariant, using combinatorial techniques based essentially on the basic theory of Alexandroff spaces together with fundamental properties of the determinant.

First, we characterize beat points in terms of the rows and columns of $X_M$.
\begin{thm}\label{thm_beat_points} Let $X=\{ x_i\}$ be a finite $T_0$-space. Then $x_i$ is an up (down) beat point if and only if there exists a row $r_j$ (column $c_j$) in $X_M$ such that $r_i-r_j=-e_i$ ($c_i-c_j=-e_i$) for some $1\leq j \leq n$. 
\end{thm}
\begin{proof}
Let us suppose that $x_i$ is an up beat point, so $\hat{F}_{x_i}$ has a minimum that we denote by $x_j$. We study cases: (i) If $x_{j,k}=0$ with $k\neq i$, then $x_j\in U_{x_k}$ so $x_i\in U_{x_k}$ and $x_{i,k}=0$ (ii) If $x_{j,k}=1$ with $k\neq i$ and  $x_{i,k}=0$, then $x_j\notin U_{x_k}$ and $x_i\in U_{x_k}$. This means that $x_i< x_k,x_j$ and $x_j\nleq x_k$, but the definition of up beat point implies that $x_i<x_j<x_k$, which entails the contradiction. Therefore, $x_{j,k}=1$ with $k\neq i$ implies $x_{i,k}=1$ (iii) By hypothesis $x_{ii}=0$ and $x_{j,i}=1$ since  $x_i\in U_{x_i}$ and $x_j\notin U_{x_i}$. Thus, $r_i-r_j=-e_i$. 

%Suppose now that $x_i$ is a down beat point, which means that $\hat{U}_{x_i}$ has a maximum that we denote by $x_j$. We study cases: (i) If $x_{ki}=0$ with $k\neq j$, then $x_k\in U_{x_i}$. Clearly $x_k< x_i<x_j$, so $x_{kj}=0$ (ii) If  $x_{ki}=1$ with $k\neq j$ and $x_{kj}=0$, then $x_k\notin U_{x_i}$ and $x_k\in U_{x_j}$. However, $x_k< x_j< x_i$ and we get a contradiction with $x_k\in U_{x_i}$ (iii) By hypothesis $x_{ij}=1$ and $x_{ii}=0$. Thus, $c_i-c_j=-e_i$.

Consider that there are two rows $r_i$ and $r_j$ satisfying $r_i-r_j=-e_i$. Therefore, $r_i$ and $r_j$ are not equal in the $i$-th coordinate, i.e., $x_{i,k}=x_{j,k}$ for every $k\neq i$ and $x_{j,i}=1$. On the other hand, $x_i<x_j$ because $x_{j,j}=0$ and $x_{i,j}-x_{j,j}=0$. We argue by contradiction, suppose there exists $x_k$ with $x_i<x_k<x_j$. Then $x_{j,k}=1$ and $x_{i,k}=0$, which gives the contradiction. This gives that $\hat{F}_{x_i}$ has a minimum given by $x_j$.

%Similarly, consider two columns $c_i$ and $c_j$ satisfying $c_i-c_j=-e_i$. Hence, $x_{ki}=x_{kj}$ for every $k\neq i$, $x_{ij}=1$ ($-1=x_{ii}-x_{ij}$) and $x_{ji}=0$ ($0=x_{ji}-x_{jj}$). Subsequently, $x_j<x_i$. If there exists $x_k$ with $x_j<x_k<x_i$, then $x_{kj}=1$ and $x_{ki}=0$, which entails a contradiction with $c_i-c_j=-e_i$. Thus, $\hat{U}_{x_i}$ has a maximum given by $x_j$. 
The other case of the result follows immediately from these arguments, Remark \ref{rem:up_son_down_opuesto} and Proposition \ref{prop:orden_opuesto_matriz}.
\end{proof}
As an immediate consequence of this result, we obtain:
\begin{cor} Let $X$ be a finite $T_0$-space. If  $c_i-c_j\neq -e_i$ and $r_i-r_j\neq -e_i$ for every $1\leq i,j\leq n$, then $X$ is not contractible .
% Then $X_M$ is not contractible if $|C(X)_i-C(X)_j|\neq 2\neq |R(X)_i-R(X)_j|$ for every $1\leq i,j\leq n$.
\end{cor}

\begin{cor}\label{cor:quitarbeatpoint_homotopy} Let $X$ be a finite $T_0$-space and $x_i\in X$ a beat point. Then $\det (X)=-\det (X\setminus x_i)$.
\end{cor}
\begin{proof}
Suppose that $x_j$ is the minimum (maximum) of $\hat{F}_{x_i}$ ($\hat{U}_{x_i}$). By Theorem \ref{thm_beat_points}, subtracting the $j$-row (column) from the $i$-row (column) and using the Laplace expansion along the $i$-row (column) one gets $\det (X)=-\det (X\setminus x_i)$. 
\end{proof}

\begin{cor}\label{cor:contractible_implies_determinant_zero} Let $X$ and $Y$ be finite $T_0$-spaces.
\begin{enumerate}
\item If $X$ and $Y$ are minimal finite spaces and $X$ is homotopy equivalent to $Y$, then $\det(X)=\det(Y)$.
\item If $X$ is homotopy equivalent to $Y$, then $|\det (X)|=|\det (Y)|$.
\item If $X$ is contractible, then $\det(X)=0$.
\end{enumerate}
\end{cor}

%\begin{cor}\label{cor:det_invariante_homotopico}
%Let $X$ and $Y$ be finite $T_0$-space. If $X$ is homotopy equivalent to $Y$, then $|\det (X)|=|\det (Y)|$.
%\end{cor}

%%\begin{cor}\label{cor:contractible_implies_determinant_zero}
%%Let $X=\{x_i \}$ be a finite $T_0$-space. If $X$ is contractible, then $\det (X)=0$.
%%\end{cor}
\begin{rem}\label{rem_contractible_no-recibe_operaciones} Notice that, given a contractible finite $T_0$-space $X$, after applying elementary row and column operations on $X_M$, one may obtain a matrix $(y_{i,j})_{i,j=1,...,n}$ such that $(y_{i,j})_{i,j=1,...,n-1}$ is the identity matrix and $y_{n,i}=y_{i,n}=0$ for every $1\leq i\leq n$.
\end{rem}

\begin{thm}\label{thm_weak_beat_point} Let $X$ be a finite $T_0$-space. If $x_i\in X$ is a weak beat point, then $\det (X)=-\det (X\setminus x_i )$.
\end{thm}
\begin{proof}
Without loss of generality, we may assume that $x_i$ is a down weak beat point and  $x_j\leq x_i$ for every $j\leq i$. That is, the submatrix $(x_{k,s})_{1\leq k,s\leq i}$ corresponds to $U_{x_i}$. Subtract $c_{i-1}$ from $c_i$ in $X_M$. Then, the resulting matrix $X_M^{(1)}=(x_{s,k}^{(1)})$ satisfies $x^{(1)}_{s,i}=0$ if $s>i$, $x^{(1)}_{i,i}=-1$ and $x^{(1)}_{s,i}=-x_{s,i-1}$ if $s<i$. We prove the last assertion. By construction $x_{s,i}=0$ if $s\leq i$ and $x_{s,i}=1$ if $s>i$. Hence, $x_{s,i}-x_{s,i-1}=-x_{s,i-1} $ for every $s<i$. Suppose that $x_{s,i-1}=0$ for some $s> i$, then $x_s<x_{i-1}$, which gives that $x_s<x_i$ and therefore the contradiction. This shows that $x_{s,i-1}=x_{s,i}$ for $s>i$. Clearly, $x_{i,i-1}=1$  because $x_{i-1}<x_{i}$, so $x_{i,i}-x_{i,i-1}=-1$.

Since $\hat{U}_{x_i}$ is contractible, by applying elementary row and column operations between the first $i-1$ rows and columns one has that $X_M^{(2)}=(x_{s,k}^{(2)})$ satisfies $(x^{(2)}_{s,k})_{1\leq s,k\leq i-2}$ is the identity matrix and $x^{(2)}_{i-1,k}=x^{(2)}_{k,{i-1}}=0$ if $k<i$ (see Remark \ref{rem_contractible_no-recibe_operaciones}). Again, by applying elementary operations using the first $i-2$ rows and columns, we obtain $X_M^{(3)}$ such that $c^{(3)}_{i}=-e_i$. Denoting $X_M^{(4)}$ to $X_{M}^{(3)}$ after removing its $i$-th column and row, we get $\det(X_M)=-\det(X_M^{(4)})$. On the other hand, note that in the process to get $X_M^{(3)}$ we do not have used the $i$-th row or column of $X_M$. Let $L$ and $R$ denote the product of the elementary matrices satisfying $X_M^{(3)}=LX_MR$. Let $L_i$ ($R_i$) denote the matrix $L$ ($R$) after removing its $i$-th column and row. Thus, $(X\setminus{x_i})_M=L_i^{-1}X_M^{(4)}R_i^{-1}$.

The result for the other case can be immediately derived from these arguments, Remark \ref{rem:up_son_down_opuesto} and Proposition \ref{prop:orden_opuesto_matriz}.

%If $x_j$ is an up weak beat point, the result follows analogously.

%%
%%If  $x_j$ is an up weak beat point, then we may repeat the same arguments used before. Again, rearranging the matrix to get the first rows corresponding to points bigger than $x_j$ and after some row and columns operations we obtain the following block matrix
%%\[\overline{X}_M=
%%\begin{pmatrix}
%%  \begin{matrix}
%%
%%  1 & 0 & 0 & \dots & 0 & 0 & 1  \\
%%  0 & 1 & 0 & \dots & 0 &0 &1 \\
%%  0 & 0 & 1 & \dots & 0 &0  & 1 \\
%%  \vdots & \vdots & \vdots & \ddots &\vdots &\vdots \\
%%  0 & 0 & 0 &\dots & 1 & 0 &1\\
%%  0 & 0 & 0 &\dots & 0 & 0 &1\\
%%  0& 0 &0 &\dots & 0  &0 & 0
%% 
%%\end{matrix}
%%  & \rvline & \begin{matrix}
%%
%%  * & * & * & \dots & * & * & *  \\
%%  * & * & * & \dots & * & * & * \\
%%  * & * & * & \dots & * &*  & * \\
%%  \vdots & \vdots & \vdots & \ddots &\vdots &\vdots  &\vdots\\
%%  * & * & * &\dots & * & * &*\\
%%  1 & 1 & 1 &\dots & 1 & 1 &1\\
%%  1& 1 &1 &\dots & 1  &1 &1 
%% 
%%\end{matrix}  \\
%%\hline
%%
%% \begin{matrix}
%%
%%   &  &  &  &  &  &   \\
%%   &  &  & * & & & \\
%%   &  &  &  &  &  & \\
%%
%% 
%%\end{matrix}   &  \rvline &  \begin{matrix}
%%
%%   &  &  &  &  &  &   \\
%%   &  &  & * & & & \\
%%   &  &  &  &  &  & \\
%%
%% 
%%\end{matrix}  
%% 
%%\end{pmatrix}
%%\]
%%To conclude we only need to repeat the previous arguments using rows instead of columns.

\end{proof}

\begin{cor}\label{cor:determinante_invariante_simple} Let $X$ and $Y$ be finite $T_0$-spaces. If $X$ and $Y$ are simple homotopy equivalent, then $|\det (X)|=|\det (Y)|$.
\end{cor}

%%%\begin{ex} Let us consider the finite model of the dunce hat introduced in \cite[Figure 7.3]{barmak2011algebraic}. This finite model is a homotopically trivial non-collapsible space of $15$ points, but its determinant is zero.
%%%\end{ex}
%Moreover, using the functor $\mathcal{K}$, this invariant also works for the notion of simple homotopy in the classical context of simplicial complexes. DESARROLLAR MÁS A FONDO USANDO EL LIBRO DE BARMAK
Given a finite simplicial complex $K$, we define $\det(K):=|\det(\mathcal{X}(K))|$.
\begin{cor}\label{cor_det_simpl_hom} Let $K$ and $L$ be two simplicial complexes. If $K$ and $L$ have the same simple homotopy type, then $\det(K)=\det(L)$.
\end{cor}
\begin{proof}
It is an immediate consequence of \cite[Corollary 3.11]{barmak2008simple} and Corollary \ref{cor:determinante_invariante_simple}.
\end{proof}

From Corollaries \ref{cor:contractible_implies_determinant_zero} \ref{cor:determinante_invariante_simple}, \ref{cor_det_simpl_hom}, and \cite[Theorem~1(b)]{chocano2025reduced}, we deduce that the reduced Euler characteristic is both a homotopy invariant and a simple homotopy invariant in the setting of finite spaces, and a simple homotopy invariant for simplicial complexes.

\subsection{Rank}
Consider $X$ to be a finite $T_0$-space. Our focus is to investigate the relationship between the rank of $X_M$, denoted as $\text{rank}(X)$, and the homotopy type of $X$. For this purpose, we introduce a new term, $\overline{\text{rank}}(X)$, which is defined as the difference between the cardinality of $X$ and the rank of $X$, i.e., $\overline{\text{rank}}(X)=|X|-\text{rank}(X)$.

\begin{thm} Let $X$ and $Y$ be finite $T_0$-spaces. If $X$ is homotopy equivalent to $Y$, then $\overline{\textnormal{rank}}(X)=\overline{\textnormal{rank}}(Y)$.
\end{thm}  
\begin{proof}
Assume that $X$ and $Y$ are homotopy equivalent. Consequently, their cores are homeomorphic. This implies that $X_M$ and $Y_M$ after performing certain row and column operations can be respectively represented as
$$\begin{pmatrix}
I_n & 0\\
0 & \overline{X}_M
\end{pmatrix}\quad \quad \quad \begin{pmatrix}
I_m & 0\\
0 & \overline{Y}_M,
\end{pmatrix}$$  
where $\overline{X}_M$ ($\overline{Y}_M$) denotes the matrix of the core of $X$ ($Y$), and $\overline{X}_M$ and $\overline{Y}_M$ are equivalent matrices (see Theorem \ref{thm:stong_cores_homeomorphic}). Therefore, we have $|X|-\text{rank}(X)=|Y|-\text{rank}(Y)$. 
\end{proof}
%\begin{ex} Let us consider $X=\{x_1,x_2,...,x_8 \}$ with $x_1<x_2>x_3<x_4>x_5<x_6>x_7<x_8>x_1$ and $Y=\{y_1,y_2,...,y \}$
%\end{ex}
\begin{cor}\label{cor_rank_simple_homotopy} Let $X$ and $Y$ be finite $T_0$-spaces. If $X$ is simple homotopy equivalent to $Y$, then $\overline{\textnormal{rank}}(X)=\overline{\textnormal{rank}}(Y)$.
\end{cor}
\begin{proof}
It is an immediate consequence of the proof of Theorem \ref{thm_weak_beat_point}.
\end{proof}

\begin{cor}
Let $X$ be a contractible finite $T_0$-space with more than one point. Then $\overline{\textnormal{rank}}(X)=1$.
\end{cor}
%\begin{proof}
%We argue by induction on the number of points of $X$. Suppose that $X$ has two points. Then $X=\{a,b\}$ with $a<b$, $$X_M=\begin{pmatrix}
%0 & 0\\
%1 & 0
%\end{pmatrix}$$
%and $\overline{\text{rank}}(X_M)=1$. Suppose $X=\{x_1,...,x_n\}$ has $n$ points. Since $X$ is contractible, it has an up (or down) beat point, let's say $x_n$. The rank of $X\setminus \{x_n \}$ is $n-2$. Moreover, $x_n$ is an up beat point, so there is a row $r_j$ in $X_M$ such that $r_n-r_i=-e_n$. Therefore, $\text{rank}(X)=|X|-1$.
%\end{proof}

%Again, as we did before, we may think about the smallest finite $T_0$-space $X$ which is non-contractible and satisfies  $\overline{\text{rank}}(X)=1$. For instance, a simple computation shows that the finite model of the projective plane $X$ considered in Example \ref{ex_projective_plane} satisfies $\overline{\textnormal{rank}}(X)=1$.

\begin{rem} Let $X$ be a finite $T_0$-space. If $Y\subseteq X$ is contractible, then $\text{rank}(X)\geq |Y|-1$.
\end{rem}

%This invariant is less rich than the determinant, given that its value can only be zero or one.
%
%\begin{thm} Let $X$ be a finite $T_0$-space. Then $\overline{\textnormal{rank}}(X)\in \{0,1 \}$. 
%\end{thm}
%\begin{proof}
%Suppose that $Y\subseteq X$ is a contractible space. Then $\text{rank}(X)\geq |Y|-1$. %In the Hasse diagram of $X$ we can consider a minimum spanning tree $Y$, this provides a subset of $X$ which is contractible and has the same cardinality as $X$. Therefore, $\text{rank}(X)\geq |X|-1$ and $\overline{\text{rank}}$ is equal to zero or one.
%\end{proof}

Finally, we adapt this notion for simplicial complexes. Given a finite simplicial complex $K$, define $\overline{\textnormal{rank}}(K):=\overline{\textnormal{rank}}(\mathcal{X}(K))$.

\begin{cor} Let $K$ and $L$ be two simplicial complexes. If $K$ and $L$ have the same simple homotopy type, then $\overline{\textnormal{rank}}(K)=\overline{\textnormal{rank}}(L)$.
\end{cor}
\begin{proof}
It is an immediate consequence of \cite[Corollary 3.11]{barmak2008simple} and Corollary \ref{cor_rank_simple_homotopy}.
\end{proof}

\subsection{Characteristic polynomial and eigenvalues}

We now turn our attention to characteristic polynomials and eigenvalues. Given a finite $T_0$-space and $X_M$, we denote its characteristic polynomial by $p_{X}(\lambda)$. In this case, these invariants are trivially not homotopy invariants because removing beat points produces polynomials with lower degree. Nevertheless, we study properties for some specific cases.
\begin{prop} Let $X$ be a finite $T_0$-space. Then $p_{X}(\lambda)=\lambda^{|X|}$ if and only if $X$ is a totally ordered set.
\end{prop} 
\begin{proof}
If $X$ is a totally ordered set, then the result follows trivially. Suppose that $p_{X}(\lambda)=\lambda^{|X|}$. Then $\textnormal{tr}(X_M^2)=0$ and we get that $X$ is a totally ordered set by Proposition \ref{prop:antichains_size_2}. 
\end{proof}
We also have that given two finite $T_0$-spaces $X$ and $Y$ that are not homeomorphic, we may get $p_{X}(\lambda)=p_Y(\lambda)$. To see this, we only need to consider the opposite order of a finite $T_0$-space. 
\begin{ex} Consider $X=\{x_1,x_2,x_3\}$ such that $x_1,x_2<x_3$ and $X^o$, that is, $X$ with $x_3<x_1,x_2$. Then  $p_{X}(\lambda)=-\lambda^3+\lambda=p_{X^o}(\lambda)$.
\end{ex}
\begin{prop} Let $X^n=\{x_i\}$ be a fence of $n$ points, i.e., $x_1<x_2>...\leq x_n$. Then $p_{X^n}(\lambda)=(-1)^n\lambda(\lambda-(n-2))(\lambda+1)^{n-2}$.
\end{prop}
\begin{proof}
Suppose $n=2$. Then it is simple to check that $p_{X^2}(\lambda)=\lambda^2$. Similarly, if $n=3$, then $p_{X^3}(\lambda)=-\lambda(\lambda-1)(\lambda+1)$. We argue by induction: suppose that $p_{X^n}(\lambda)=(-1)^n\lambda(\lambda-(n-2))(\lambda+1)^{n-2}$. By construction, $X^{n+1}$ is obtained by adding a point $x_{n+1}$ to $X^n$. If $n$ is an odd number, then $x_{n+1}>x_n$. Therefore, the $(n+1)$-th row (column) of $X^{n+1}_M$ is $(1,1,...,1,0)$ ($(1,1,...,1,0,0)^\intercal$). To compute $\det(X_M^{n+1}-\lambda I)$ subtract the $n$-row from $(n+1)$-th row to get that the resulting row is $(0,0,...,0,\lambda,-1-\lambda)$. Hence,
$$p_{X^{n+1}}(\lambda)=-1(\lambda+1)p_{X^n}(\lambda)-\lambda |A|,$$
where $A$ is a $n\times n$ matrix such that $(a_{i,j})_{i,j=1,...,n-1}=X_M^{n-1}-\lambda I$ and $a_{i,n}=a_{n,i}=1$ for $i=1,...,n$. To compute $|A|$, subtract the $n$-row from each row. Using the Laplace expansion along the $n$-th column, we get $|A|=|B|$, where $B$ is a tridiagonal matrix. Let $a_i$ denote the elements of the main diagonal, so $a_i=-\lambda-1$ for $i=1,...,n-1$. The upper diagonal (lower diagonal), denoted by $b_i$ ($c_i$), is given by $b_i=0$ ($c_i=-1$) if $i$ is odd and $b_i=-1$ ($c_i=0$) if $i$ is even where $i=1,...,n-2$. Hence, $|B|=(-\lambda-1)^{n-1}$. So putting it all together, one has 
\begin{align*}
p_{X^{n+1}}(\lambda)&=(-1)^{n+1}\lambda(\lambda-(n-2))(\lambda+1)^{n-1}+(-1)^{n} \lambda(\lambda+1)^{n-1}\\&=(-1)^{n+1}\lambda(\lambda+1)^{n-1}(\lambda-(n-1)).
\end{align*}
We may repeat similar arguments for the case that $n$ is an even number, in that case, we need to do column operations instead of row operations.
\end{proof}

This result tells that we might not expect to get results of divisibility between the characteristic polynomials of spaces that are homotopic.  In addition, we can produce contractible spaces with spectral radius as big as we want.

\begin{prop} Let $X$ be a finite $T_0$-space. If $x\in X$ is a maximum (minimum), then $p_{X}(\lambda)=-\lambda p_{X\setminus x}(\lambda)$.
\end{prop}
\begin{proof}
Suppose that $X$ has a maximum (minimum). Then $X_M$ has a null column (or row), so $p_X(\lambda)=-\lambda p_{X\setminus x}(\lambda)$.
\end{proof}

\begin{rem} A deeper study of the eigenvalues of symmetric matrices related to those considered here is presented in \cite{chocano2025quadratic}, where embedding problems between finite $T_0$-spaces are also investigated.
\end{rem}

We define the characteristic polynomial of a finite simplicial complex $K$ as the characteristic polynomial of the matrix $\mathcal{X}(K)_M$. Similarly, we define its eigenvalues as those of $\mathcal{X}(K)_M$.

\section{Other aspects of the matrices from $\mathcal{M}$}\label{sec:other_aspects}

In this section, we study combinatorial aspects and analyze other questions about the matrices considered.

\subsection{Determinant of the sum of a topological space with the identity}

We can give a topological interpretation to $\det(X_M+I_{|X|\times |X|})$ for each finite $T_0$-space $X$. Recall that in general, given two square matrices $A$ and $B$ of dimension $n$, $\det(A+B)=\det(A)+\det(B)+\sum_{i=1}^{n-1} \Gamma_n^i\det(A/B^i)$, where $\Gamma_n^i\det(A/B^i)$ is defined as a sum of the combination of determinants, in which the $i$ rows of $A$ are substituted by the corresponding rows of $B$ (see \cite{detAmasB}). So in the current context, $\Gamma_{n}^i(X_M/I_{|X|\times |X|}^i)=\sum_{j=1}^n \det (X \setminus \{x_{j_1},...,x_{j_i}\})$ with $j_l\in \{1,...,n\}$ for every $j$ and $l$ and  $j_l\neq j_k$ for each $j=1,...,i$ and $l\neq k$. Therefore, $\Gamma_{n}^i(X_M/I_{|X|\times |X|}^i)$ is the sum of the determinants of the topological spaces obtained after removing $i$ points from $X$ and  $\det(X_M+I_{|X|\times |X|})=\det(X_M)+\sum_{Y\subset X}\det(Y_M)$. Particularly, we may interpret $\Gamma_{n}^{i}(X_M/I_{|X|\times |X|}^{i})$ as a sum of the different ways that appear in $X$ spaces of $n-i$ points with the same relations defined on $X$, that is, the number of different induced posets of $n-i$ points in $X$.

%We introduce notation: given a finite $T_0$-space,  $A_i$ denotes the number of antichains of $i$ points in $X$, $L^j_i$ denotes the number of the different ways to embed in $X$  a poset of $j$ points that is homotopy equivalent to an antichain of $i$ points, $S^i_1$ denotes the different ways to embed in $X$ a poset has $i$ points and is homotopy equivalent to the minimal finite model of the circle, $S^i_{1,...^k,1}$  denotes the different ways to embed in $X$ a poset that has $i$ points and is homotopy equivalent to a finite model of the wedge sum of $k$ circles, $S^i_{1,...^k,1} \sqcup {j}$ denotes the different ways to embed in $X$ a poset has $i$ points and is homotopy equivalent to the disjoint union of a finite model of the wedge sum of $k$ circles with an antichain of $j$ points. Following the previous comments and this notation we get from the classification of finite $T_0$-spaces the following result:
%PROBLEMA CON LOS NUMEROS, LAS CUENTAS NO SALEN 
We present in the following result formulas for $\Gamma_{n}^i(X_M/I_{|X|\times |X|}^i)$ when $i$ is close to be $n$. The proof is omitted as it is straightforward.
\begin{prop} Let $X$ be a finite $T_0$-space of $n$ points. Then 
\begin{itemize}
\item $\Gamma_{n}^{n-1}(X_M/I_{|X|\times |X|}^{n-1})=0$.
\item $\Gamma_{n}^{n-2}(X_M/I_{|X|\times |X|}^{n-2})=-A_2$, where $A_2$ denotes the number of antichains of size $2$ in $X$.
\item $\Gamma_{n}^{n-3}(X_M/I_{|X|\times |X|}^{n-3})=L^3_2+2A_3$, where $A_3$ denotes the number of antichains of size $3$ in $X$ and $L^3_2$ denotes the number of posets in $X$ that are homeomorphic to the poset of three points $Y=\{y_1,y_2,y_3 \}$ defined by $y_1<y_2$.
%\item $\Gamma_{n}^{n-4}(X_M/I_{|X|\times |X|}^{n-4})=-2L^4_3-L^4_2+S_1-3A_4$.
%\item $\Gamma_{n}^{n-5}(X_M/I_{|X|\times |X|}^{n-5})=3L^5_4+2L^5_3+L^5_2-2S_{1,1}-S_1+4A_5$.
%\item $\Gamma_{n}^{n-6}(X_M/I_{|X|\times |X|}^{n-5})=-3L^6_4-2L^6_3-L^6_2+-4S_{1,1,1,1}+3S_{1,1,1}+2S_{1,1}+S_1-S_2-S_1^{..}+S_{1,1}^{.}-5A_6$. 
\end{itemize}
\end{prop} 

It is clear that $\det(X_M+I_{|X|\times |X|})$ is not a homotopy invariant. For example, this number does not coincide for a fence and a totally ordered set. In fact, we have that this number can only be zero or one:
\begin{prop} Let $X$ be a finite $T_0$-space. 
\begin{enumerate}
\item Then $X$ is a totally ordered set if and only if $\det(X_M+I_{|X|\times |X|})=1$.
\item Then $\det(X_M+I_{|X|\times |X|})=0$ if and only if $X$ is not a totally ordered set. 
\end{enumerate}
\end{prop}
\begin{proof}
(1) If $X$ is a totally ordered set, the result follows easily. Assume $\det(X_M+I_{|X|\times |X|})=1$. We argue by contradiction. Suppose that $X$ is not a totally ordered set. If $X$ has two maximal or minimal points, then it is clear that $\det(X_M+I_{|X|\times |X|})$ should be zero because we obtain two rows or two columns that are equal. If $X$ has a maximum (or minimum) $x$, then we use the Laplace expansion along the column (or row) that corresponds to $x$. The resulting determinant is equal to the determinant of $(X\setminus x)_{M}+I_{|X|-1\times |X|-1}$. If $X\setminus x$ has two maximal or minimal points we have a contradiction. If $X\setminus$ has a maximum or minimum, we can repeat the previous argument. Since $X$ is not a totally ordered set, after some steps we obtain that $\det(X_M+I_{|X|\times |X|})=\det(X\setminus \{x_{i_1},x_{i_2},...,x_{i_k} \}_M+I_{|X|-k\times |X|-k})=0$, which entails the contradiction.

(2) Knowing that if $X$ has two maximal or minimal elements then $\det(X_M+I_{|X|\times |X|})=0$, we only need to repeat the previous arguments, removing the maximum (or minimum) from $X$ if it exists.
\end{proof}

\subsection{Group actions}
We assume that the finite $T_0$-spaces of this subsection admit the free action of a finite group and obtain their matrices.
\begin{thm}\label{thm:group_action}
Consider a finite $T_0$-space $X$ and a finite group $G=\{g_1,...,g_m\}$, where $g_1$ is the identity element. Suppose that $X$ admits a $G$-free action. Then $X_M$ is a block matrix $\{A_{i,j}\}_{i,j=1,...,m}$, where $A_{i,j}$ has dimension $\frac{|X|}{m} \times \frac{|X|}{m}$, satisfying that $A_{i,j}=A_{1,s}$, where $g_s=g_jg_i^{-1}$, for every $i,j=1,...,m$. 
\end{thm}
\begin{proof}
Consider a fundamental domain $D=\{x_0,...,x_{m-1} \}\subset X$. Label the elements of $X$ as follows $\{x_0^{g_1},...,x_{m-1}^{g_1},x_{0}^{g_2},...,x_{m-1}^{g_2},...,x_{0}^{g_m},...,x_{m-1}^{g_m} \}$, where $g_i(x_k^{g_l})=x_k^{g_lg_i}$. We have that $x_{l}^{g_i}\leq x_{k}^{g_j}$ if and only if $g_i(x_l^{g_1})\leq g_i(x_k^{g_{j}g_i^{-1}})$ if and only if $x_l^{g_1}\leq x_k^{g_jg_i^{-1}}$.  From this, we deduce the result.
\end{proof}
Note that in the hypothesis of Theorem \ref{thm:group_action}, the matrix $X_M$ is a Latin block square matrix. For finite $T_0$-spaces that admit a $\mathbb{Z}_2$-free action, we have the following simple description.

%%\begin{thm}
%%Let $X$ be a finite $T_0$-space. suppose that $X$ admits a $\mathbb{Z}_m$-free action for some integer $m\geq 2$. Then $X_M$ is a block matrix $\{A_{ij}\}_{i,j=1,...,m}$, where $A_{ij}$ has dimension $\frac{|X|}{m} \times \frac{|X|}{m}$, satisfying that $A_{ij}=A_{1(s+1)}$, where $s\equiv j-i\ (\text{mod} \ m)$, for every $i,j=1,...,m$. Particularly, $\det()$
%%\end{thm}
%%\begin{proof}
%%Consider a fundamental domain $D=\{x_0,...,x_{m-1} \}\subset X$. Label the elements of $X$ as follows $\{x_0^0,...,x_{m-1}^0,x_{0}^1,...,x_{m-1}^1,...,x_{0}^{m-1},...,x_{m-1}^{m-1} \}$, where $j(x_k^{l})=x_k^{s}$ with $s\equiv j+l \ (\text{mod} \ m)$. We have that $x_{l}^{i}\leq x_{k}^j$ if and only if $i(x_l^0)\leq i(x_k^{j-i})$ if and only if $x_l^0\leq x_k^{j-i}$.  From this, we deduce the result.
%%\end{proof}

\begin{cor}
Let $X$ be a finite $T_0$-space. Suppose that $X$ admits a $\mathbb{Z}_2$-free action. Then $X_M$ is a block matrix $(A_{i,j})_{i,j=1,2}$ such that $A_{1,1}=A_{2,2}$ and $A_{1,2}=A_{2,1}$, where $A_{i,j}$ has dimension $\frac{|X|}{2}\times \frac{|X|}{2}$. Particularly, $\det (X_M)=\det(A_{1,1}+A_{1,2})\det(A_{1,1}-A_{1,2})$.
\end{cor}

Combining Proposition \ref{prop:sumas_elementos_filas_columnas} and Theorem \ref{thm:group_action}, we get:
\begin{prop} Let $X$ be a finite $T_0$-space that admits a $G$-free action. Then $\sum_{i=1}^{|X|}U_{x_i}$ and $\sum_{i=1}^{|X|}F_{x_i}$ are multiples of $|G|$.
\end{prop} 

It is easy to deduce that the minimal finite model of the $n$-sphere $S^n$ satisfies $\sum_{x\in S^n}|U_x|=\sum_{x\in S^n}|F_x|=2(\sum_{i=0}^n 2i+1)=2 (n+1)^2$. This is in concordance with the fact that every minimal finite model of a sphere  admits a $\mathbb{Z}_2$-free action (the antipodal map), that interchanges the points of the same height.  %This gives that if $n$ is odd, then $S^n$ does not admit a free action of a finite group $G$ of odd order. Similarly, it is easy to compute that $\sum_{x\in S^n\times S^m}|U_x|=\sum_{x\in S^n\times S^m}|F_x|=2(n+1)^2((m+1)^2+(2m+2)^2)$.

\subsection{Sums of columns and rows as topological invariants}

%BUSCAR MANERAS DE OBTENER CLASES DE EQUIVALENCIAS O NUMEROS
%https://oeis.org/A002724 SUCESION DE LAS BINARIAS NXN
Let $X$ be a finite $T_0$-space. We define $R(X)$ as the sum of the row entries in $X_M$, $R(X)=\sum_{i=1}^n r_i$. Similarly, we define $C(X)$ as the sum of the column entries in $X_M$, $C(X)=\sum_{i=1}^n c_i$. In addition, we can consider the sum of all entries in $X_M$, denoted by $\sum(X)$. The following result is a direct consequence of these definitions.

\begin{prop} Let $X$ and $Y$ be finite $T_0$-spaces. If $X$ is homeomorphic to $Y$, then $R(X)$, $C(X)$ and $R(Y)$ ,$C(Y)$ are equal respectively up to some permutations of their entries, and $\sum(X)=\sum(Y)$. 
\end{prop}

It is clear that the previous invariants do not classify finite $T_0$-space. For instance, consider $X=\{x_1,x_2,x_3,x_4\}$ where $x_1<x_3,x_4$; $x_2<x_3,x_4$ and $x_3<x_4$, and $Y=\{y_1,y_2,y_3,y_4\}$ where $y_1<y_3,y_4$; $y_2<y_4$ and $y_3<y_4$. $R(X)=R(Y)$, but $X$ is not homeomorphic to $Y$. This invariant only classifies completely finite spaces of at most $3$ points. Furthermore, in \cite{quasis1966sharp}, there is an example of two finite $T_0$-spaces of $6$ points that are not homeomorphic, but they agree on $R(\cdot)$ and $C(\cdot)$. On the other hand, 

\begin{thm} Let $X$ and $Y$ be minimal finite spaces. If $X$ is homotopy equivalent to $Y$, then $R(X)=R(Y)$, $C(X)=C(Y)$ and $\sum(X)=\sum(Y)$.
\end{thm}

Observe that the spaces in the previous examples are not minimal finite spaces. Consider $X=\{x_1,x_2,x_3,x_4,x_5,x_6,x_7,x_8 \}$ where $x_1<x_5>x_2<x_6>x_3<x_7>x_4<x_8>x_1$ and $Y=\{y_1,y_2,y_3,y_4,y_5,y_6,y_7,y_8 \}$ where $y_1<y_5>y_2<y_6>y_1$ and $y_3<y_7>y_4<y_8>y_3$. It is evident that both $X$ and $Y$ are minimal finite spaces because they do not have beat points. Specifically, $X$ is a finite model of the circle and $Y$ is a finite model of the disjoint union of two circles. Although $X$ is not homeomorphic to $Y$,  $R(X)=R(Y)$, $C(X)=C(Y)$ and $\sum(X)=\sum(Y)$. %As discussed in previous sections, we can propose the following problem: find the smallest two finite (connected) minimal finite spaces $X$ and $Y$ such that $R(X)=R(Y)$, $C(X)=C(Y)$ and $X$ is not homotopy equivalent to $Y$. 

\bibliography{bibliografia}
\bibliographystyle{plain}

\newcommand{\Addresses}{{ additional braces for segregating \footnotesize
  \bigskip
  \footnotesize

  \textsc{ P.J. Chocano, Departamento de Matemática Aplicada,
Ciencia e Ingeniería de los Materiales y
Tecnología Electrónica, ESCET
Universidad Rey Juan Carlos, 28933
Móstoles (Madrid), Spain}\par\nopagebreak
  \textit{E-mail address}:\texttt{pedro.chocano@urjc.es}

}}

\Addresses

\end{document}